\documentclass[12pt]{article}

\usepackage[dvips]{graphicx}
\usepackage{amssymb}
\usepackage{amsmath}
\usepackage{amsthm}

\newtheorem{theorem}{Theorem}
\newtheorem{lemma}[theorem]{Lemma}
\newenvironment{keywords}{{\bf Keywords:}}

\newcommand{\R}{{\mathbb R}}
\newcommand{\Q}{{\mathbb Q}}
\newcommand{\N}{{\mathbb N}}

\begin{document}

\title {On profinite spaces of orderings}

\author{Pawe\l \ G\l adki
\thanks{Corresponding author: phone +48 32 359 2228, fax +48 32 258 2976 } 
\thanks{This research was conducted while the first author was employed as a visiting assistant professor by the Department of Mathematics, University of California, Santa Barbara, CA, USA, 93106.}\\
{\small \sc Institute of Mathematics,} \\ 
{\small \sc University of Silesia,}\\
{\small \sc ul. Bankowa 14,}\\
{\small \sc Katowice, Poland, 40-007}\\
{\small \tt pawel.gladki@us.edu.pl}
\and Bill Jacob\\
{\small \sc Department of Mathematics,} \\ 
{\small \sc University of California, Santa Barbara,}\\
{\small \sc Santa Barbara, CA, USA, 93106}\\
{\small \tt jacob@math.ucsb.edu}
}

\date{}

\maketitle

\begin{abstract}
In this paper we present the following two results: we give an explicit description of the space of orderings $(X_{\Q(x)}, G_{\Q(x)})$ as an inverse limit of finite spaces of orderings and we provide a new, simple proof of the fact that the class of spaces of orderings for which the pp conjecture holds true is closed under inverse limits. We discuss how these theorems interact with each other, and explain our motivation to look into these problems.
\end{abstract}


\section{Introduction and notation}

The theory of abstract spaces of orderings was developed by Murray Marshall in the 1970s (see the monograph \cite{Marshall1996}), and provides a convenient framework for studying orderings of fields and the reduced theory of quadratic forms. Spaces of orderings also occur in a natural way in other, more general settings: as maximal orderings on semi-local rings, as orderings on skew fields, or as orderings on ternary fields. The axioms for spaces of orderings have been also generalized in various directions -- to quaternionic schemes, to spaces of signatures of higher level, or to abstract real spectra that are used to study orderings on commutative rings.

Among all spaces of orderings, profinite spaces, that were introduced in \cite{Marshall1979}, are of considerable interest, and numerous questions regarding such spaces remain open. In particular, it is important to recognize which spaces of orderings are inverse limits of finite spaces \cite[Question 1]{Marshall1979}, and in this paper we provide some insights into this question. As the first result, we show that the space of orderings of the field $\Q(x)$ is profinite. The proof presented here is, in a way, constructive, and gives a ``geometric'' meaning to the result. Then, in Section 3, we exhibit a few spaces that are not profinite -- we do it by means of the pp conjecture.
To be more specific, we give another proof of a theorem previously proven by Astier and Tressl in \cite{AstTre} that the class of spaces for which the pp conjecture holds true is closed with respect to inverse limits -- in particular, we show that the conjecture is valid for profinite spaces -- and then we recall familiar examples of spaces for which the conjecture fails. This relates to another long standing question, namely if every abstract space of orderings is realized as a subspace of a space of orderings of a field: it seems likely that spaces of orderings exist which are not so realized, but to date no such examples are known.

Throughout this paper by $(X, G)$ we understand a space of orderings in the sense of \cite[pp. 21-22]{Marshall1996}: $X$ is a nonempty set, $G$ is a subgroup of $\{1, -1\}^X$, which contains the constant function $-1$, separates points of $X$, and satisfies some extra axioms -- see \cite{Marshall1996} for details. $X$ can be also viewed as a subset of the character group $\chi(G)$ (here by characters we mean group homomorphisms $x: G \rightarrow \{-1, 1\}$) via a natural embedding $X \hookrightarrow \chi(G)$ obtained by identifying $x \in X$ with the character $G \ni a \mapsto a(x) \in \{-1, 1\}$.

The theory of spaces of orderings is parallel to the theory of reduced special groups in the language $L_{SG}$ of special groups (see \cite{DicMir00} for a full list of axioms). We also note that spaces of orderings are essentially the same thing as real reduced multifields (see \cite{Mar06}), and we could use the language of rings with multivalued addition instead of the language of special groups.

If $(X, G)$ is a space of orderings, and $G_0$ is a subgroup of $G$, we denote by $X_0$ the set of all characters from $X$ restricted to $G_0$. In the case when $(X_0, G_0)$ is a space of orderings, following \cite{Marshall1979} we call it a quotient space of $(X, G)$ -- otherwise, in general, we call it a {\em quotient structure}. As noted in  \cite{Marshall1979}, if $(X_0, G_0)$ is a space of orderings then $-1\in G_0$.

For a space of orderings $(X, G)$ and two elements $a, b \in G$ we define the binary form as the formal tuple $(a,b)$. The value set of this binary form is then defined as
$$D_X(a,b) = \{c \in G: \forall x \in X \ (a(x) = c(x) \vee b(x) = c(x))\}.$$

By a morphism $F$ from a space of orderings $(X_1, G_1)$ to a space of orderings $(X_2, G_2)$ we understand a function $F: X_1 \rightarrow X_2$ such that
$$\forall b \in G_2 \ (b \circ F \in G_1).$$
A morphism $F:(X_1, G_1) \rightarrow (X_2, G_2)$ defines a group homomorphism $F^*: G_2 \rightarrow G_1$ given by $F^*(b) = b \circ F$ which also satisfies the condition
$$\forall b_1, b_2, b_3 \in G_2 \ [(b_1 \in D_{X_2}(b_2, b_3)) \Rightarrow (F^*(b_1) \in D_{X_1}(F^*(b_2), F^*(b_3)))],$$
and thus becomes a morphism of reduced special groups. Clearly, a bijective morphism will be called an isomorphism, and we shall write $(X_1, G_1) \cong (X_2, G_2)$ to indicate that the two spaces of orderings are isomorphic.

An inverse system of spaces of orderings is a triple consisting of: (1) a directed set $(I, \succeq)$, (2) spaces of orderings $(X_i, G_i)$, $i \in I$, and (3) morphisms $F_{ij}: (X_i, G_i) \rightarrow (X_j, G_j)$ defined for $i \succeq j$, $i, j \in I$, such that (a) $F_{ij} (X_i) = X_j$, which implies that
$F^*_{ij}:G_j\rightarrow G_i$ is injective, and (b) $F_{ik} = F_{jk} \circ F_{ij}$, for $i \succeq j \succeq k$, $i, j, k \in I$.

Clearly, an inverse system $(I, (X_i, G_i), F_{ij})$ of spaces of orderings automatically defines both a direct system of groups $(I, G_i, F_{ij}^*)$, and an inverse system of character sets $(I, X_i, F_{ij})$. Further, if we let $G = \underrightarrow{\lim } G_i$, and $X = \underleftarrow{\lim } X_i$, then $(X, G)$ is a space of orderings that is called the inverse limit of the given inverse system and denoted by $\underleftarrow{\lim } (X_i, G_i)$ (\cite[Theorem 4.3]{Marshall1979}). For a fixed $j \in I$ we will denote by $\pi_j$ the projection $\pi_j: X \rightarrow X_j$ such that $\pi_j = F_{ij} \circ \pi_i$, for $i \succeq j$, $i \in I$, and by $\gamma_j$ the injection $\gamma_j: G_j \rightarrow G$ such that $\gamma_j = \gamma_i \circ F_{ij}^*$, for $i \succeq j$, $i \in I$. Since, in fact, $G = \bigcup_{i \in I} G_i$, we will use the same symbol $a$ for an element $a \in G_i$ and its image $a \in \gamma_i(G_i) \subset G$. A space of orderings which is an inverse limit of finite spaces of orderings will be called profinite.

We shall say that $(X, G)$ is the direct sum of the spaces of orderings $(X_i, G_i)$, $i \in \{1, \ldots, n\}$, denoted $(X, G) = \coprod_{i=1}^n (X_i, G_i) = (X_1, G_1) \sqcup \ldots \sqcup (X_n, G_n)$, if $X$ is the disjoint union of the sets $X_1, \ldots, X_n$, and $G$ consists of all functions $a: X \rightarrow \{-1, 1\}$ such that $a|_{X_i} \in G_i$, $i \in \{1, \ldots, n\}$. In this case $G = G_1 \oplus G_2 \oplus \ldots \oplus G_n$, with the role of the distinguished element $-1$ played by $(-1, -1, \ldots, -1)$. Further, we shall say that $(X, G)$ is a group extension of the space of orderings $(\overline{X}, \overline{G})$, if $G$ is a group of exponent 2, $\overline{G}$ is a subgroup of $G$, and $X = \{x \in \chi(G): x|_{\overline{G}} \in \overline{X}\}$. Since $G$ decomposes as $G = \overline{G} \times H$, we shall also write $(X,G) = (\overline{X}, \overline{G}) \times H$ to denote group extensions. Both direct sums and group extensions are spaces of orderings themselves (\cite[Theorem 4.1.1]{Marshall1996}), and every finite space of orderings is built up, in an essentially unique way, from one element spaces, using repeatedly the direct sum and group extension operations. The non-uniqueness arises only from the exceptional property of the two element space, that can be viewed either as the direct sum of two one element spaces, or as a group extension of a one element space (\cite[Theorem 4.2.2]{Marshall1996}). 

For any space of orderings $(X, G)$, $X$ has a natural topology, namely the one introduced by the family of subbasic clopen Harrison sets of the form:
$$H_X(a) = \{x \in X: a(x) = 1\},$$
for a given $a \in G$. $X$ endowed with this topology is a Boolean space (that is compact, Hausdorff, and totally disconnected) (\cite[Theorem 2.1.5]{Marshall1996}). A subset $Y \subset X$ is called a subspace of $(X, G)$, if $Y$ is expressible in the form $\bigcap_{a \in S} H_X(a)$, for some subset $S \subset G$. For any subspace $Y$ we will denote by $G|_Y$ the group of all restrictions $a|_Y$, $a \in G$. The pair $(Y, G|_Y)$ is a space of orderings itself (\cite[Theorem 2.4.3]{Marshall1996}).

If $G$ is any multiplicative group of exponent 2 with distinguished element $-1$, we set $X = \{x \in \chi(G): x(-1) = -1\}$ and call the pair $(X,G)$ a fan. Any fan is also a space of orderings (\cite[Theorem 3.1.1]{Marshall1996}). If $(X, G)$ is a space of orderings, by a fan in $(X, G)$ we understand a subspace ${\cal F}$ such that the space $({\cal F}, G|_{{\cal F}})$ is a fan. The stability index of a space of orderings $(X, G)$ is the maximum $n$ such that there exists a fan ${\cal F} \subset X$ with $|{\cal F}| = 2^n$ (or $\infty$ if there is no such $n$).

For a space of orderings $(X, G)$ we define the connectivity relation $\sim$ as follows: if $x_1, x_2 \in X$, then $x_1 \sim x_2$ if and only if either $x_1 = x_2$ or there exists a four element fan ${\cal F}$ in $(X, G)$ such that $x_1, x_2 \in {\cal F}$. The equivalence classes with respect to $\sim$ are called the connected components of $(X, G)$ (see \cite[p. 66]{Marshall1996}; the fact that $\sim$ is, in fact, an equivalence relation follows from \cite[Theorem 4.6.1]{Marshall1996}). Consequently, if $(X, G)$ is a finite space of orderings, and $X_1, \ldots, X_n$ are its connected components, then $(X, G) = (X_1, G|_{X_1}) \sqcup \ldots \sqcup (X_n, G|_{X_n})$, where $(X_i, G|_{X_i})$, $i \in \{1, \ldots, n\}$, are either one element spaces or proper group extensions of some other spaces (\cite[Theorem 4.2.2]{Marshall1996}).

For a formally real field $k$, we will denote by $X_k$ the set of all orderings of $k$, and by $G_k$ the multiplicative group $k^*/(\Sigma k^2)^*$ of all classes of sums of squares of $k^*$. $G_k$ is naturally identified with a subgroup of $\{-1,1\}^{X_k}$ via the homomorphism
$$k^* \ni a \mapsto \overline{a} \in \{-1, 1\}^{X_k}, \mbox{ where } \overline{a}(\sigma) = \left\{ \begin{array}{ll} 1, & \mbox{ if } a \in \sigma,\\
-1, & \mbox{ if } a \notin \sigma, \end{array} \mbox{ for } \sigma \in X_k, \right.$$
whose kernel is $(\Sigma k^2)^*$, and $(X_k, G_k)$ is a space of orderings (\cite[Theorem 2.1.4]{Marshall1996}). For the sake of simplicity we shall denote by the same symbol $a$ both an element $a \in k^*$, a class of sums of squares $a \in k^*/(\Sigma k^2)^*$, and a function $a \in \{-1, 1\}^{X_k}$. Also, for an abstract space of orderings $(X, G)$ we will usually denote elements of the set $X$ by small letters $x, y, z, \ldots$, while for a space of orderings $(X_k, G_k)$ of a field $k$ we shall denote orderings from the set $X_k$ by small Greek letters $\sigma, \tau, \upsilon, \ldots$

Recall that a preordering of a formally real field $k$ is a subset $T \subset k$ such that $T + T \subset T$, $T \cdot T \subset T$, and $k^2 \subset T$. Every proper preordering can be extended to an ordering (\cite[Theorem 1.1.1]{Marshall1996}), and the set of all orderings $P$ of $k$ such that $P \supset T$, for a given preordering $T$, will be denoted by $X_T$. Also, the group $k^*/T^*$ will be denoted by $G_T$. Subspaces of the space of orderings $(X_k, G_k)$ have the form $(X_T, G_T)$, where $T$ is some preordering in $k$ (\cite[page 33]{Marshall1996}). 

We shall describe elements of $X_{\Q(x)}$ in some more detail (see, for example, \cite[Notation 1.4]{DicMarMir}). Each irreducible polynomial $p \in \Q[x]$ with real roots $\alpha_1 < \ldots < \alpha_n$, $n \geq 1$, gives rise to $2n$ orderings of $\Q(x)$, namely $\sigma_j^-$ and $\sigma_j^+$, $j \in \{1, \ldots, n\}$, defined as follows: for $a \in \Q(x)^*$ and $j \in \{1, \ldots, n\}$, $a \in \sigma_j^-$ if and only if, for some $\epsilon > 0$, $a$ is strictly positive on the interval $(\alpha_j - \epsilon, \alpha_j)$, and $a \in \sigma_j^+$ if and only if, for some $\epsilon > 0$, $a$ is strictly positive on the interval $(\alpha_j, \alpha_j + \epsilon)$. Similarly, we define two orderings ``at infinity'' $\infty^-$ and $\infty^+$: for $a \in \Q(x)^*$, $a \in \infty^-$ if and only if, for some $\xi \in \Q$, $a$ is strictly positive on the interval $(-\infty, \xi)$, and $a \in \infty^+$ if and only if, for some $\xi \in \Q$, $a$ is strictly positive on the interval $(\xi, +\infty)$. Finally, for each transcendental number $\zeta \in \R$, we consider the embedding $\Q(x) \hookrightarrow \R$ defined by $x \mapsto \zeta$, that induces an ordering by taking the counterimage of all nonnegative reals. These are precisely all the elements of $X_{\Q(x)}$. The four-element fans in $(X_{\Q(x)}, G_{\Q(x)})$ are the sets $\{\sigma_i^-, \sigma_i^+, \sigma_j^-, \sigma_j^+\}$, for $1 \leq i < j \leq n$, and for some irreducible polynomial having $n \geq 2$ real roots (\cite[Notation 1.4]{DicMarMir}).

\section{Representation of $(X_{\Q(x)}, G_{\Q(x)})$ as a profinite space}

We prove here the following:

\begin{theorem} The space of orderings $(X_{\Q(x)}, G_{\Q(x)})$ is profinite.
\end{theorem}

\begin{proof} It suffices to show that for a given finite subset $\{p_1, \ldots, p_m\} \subset G_{\Q(x)}$ there exists a finite quotient space $(X_0, G_0)$ of $(X_{\Q(x)}, G_{\Q(x)})$ such that 
$$p_1, \ldots, p_m \in G_0$$ (\cite[Remark 5.5]{Marshall1979}). Thus let $\{p_1, \ldots, p_m\} \subset G_{\Q(x)}$, and let $\widetilde{G} = \langle p_1, \ldots, p_m \rangle$ be a subgroup of $G_{\Q(x)}$ generated by the elements $p_1, \ldots, p_m$. Without loss of generality we may assume that $p_1, \ldots, p_m$ are square free polynomials, and replacing, if necessary, the set $\{p_1, \dots, p_m\}$ with the set of all irreducible factors of $p_1, \ldots, p_m$, we may also assume that the sets of real roots of polynomials $p_1, \ldots, p_m$ are pairwise disjoint.

Let $\alpha_{k,1} < \ldots < \alpha_{k,n_k}$ denote all the real roots of $p_k$, $k \in \{1, \ldots, m\}$.
 If the total number of roots of all $p_k$ is $N$, we separate them with $N+1$ rational lines -- we denote the two lines neighboring $\alpha_{k,j}$ by $\ell_{k,j}^-$, $\ell_{k,j}^+$, where $\ell_{k,j}^- = x - \xi_{k,j}^- \in \Q[x]$, $\ell_{k,j}^+ = x - \xi_{k,j}^+ \in \Q[x]$, and $\xi_{k,j}^- < \xi_{k,j}^+$. Note that if $\alpha_{k',j'}$ and $\alpha_{k'',j''}$, $\alpha_{k',j'} < \alpha_{k'',j''}$, are two consecutive roots from the set $\{ \alpha_{k,j}: k \in \{1, \ldots, m\}, j \in \{1, \ldots, n_k\}\}$, for some $k',k'' \in \{1, \ldots, m\}$, $j' \in \{1, \ldots, n_{k'}\}$, $j'' \in \{1, \ldots, n_{k''}\}$, then 
$$\ell_{k',j'}^+ = \ell_{k'',j''}^-.$$

We proceed with the consruction of a finite quotient of $(X_{\Q(x)}, G_{\Q(x)})$ as follows: start with roots $\alpha_{1,1} < \ldots < \alpha_{1,n_1}$ of the polynomial $p_1$ and consider the direct sum of one element spaces
\begin{align*}
\lefteqn{ \left( \{\widehat{\sigma_{1,1}}\}, \langle \ell_{1,1}^- \cdot \ell_{1,1}^+\rangle \right) \sqcup \left( \{\widehat{\sigma_{1,2}}\}, \langle \ell_{1,2}^- \cdot \ell_{1,2}^+\rangle \right) \sqcup \ldots} \\
& \ldots \sqcup \left( \{\widehat{\sigma_{1,n_1-1}}\}, \langle \ell_{1,n_1-1}^- \cdot \ell_{1,n_1-1}^+\rangle \right) \sqcup 
\left( \{\widehat{\sigma_{1,n_1}}\}, \langle \ell_{1,n_1}^- \cdot \ell_{1,n_1}^+\rangle \right).
\end{align*}
Here by $\widehat{\sigma_{1,j}}$ we mean the unique ordering that makes the element $\ell_{1,j}^- \cdot \ell_{1,j}^+$ of the group $\langle \ell_{1,j}^- \cdot \ell_{1,j}^+ \rangle$ negative, $j \in \{1, \ldots, n_1\}$. Now take the group extension of the above space by the two element group generated by the element $h_1$ obtained by multiplying $p_1$ by the product $\ell_{1,1}^+ \cdot \ell_{1,2}^+ \cdot \ldots \cdot \ell_{1, n_1}^+$ and possibly by the element $\epsilon = -1$, so that the resulting polynomial of even degree is negative only on either the interval $(\xi_{1,j}^-, \alpha_{1,j})$ or the interval $(\alpha_{1,j}, \xi_{1,j}^+)$, but not both, for each $j \in \{1, \ldots, n_1\}$, and positive elsewhere. Denote this new space by $(X_1, G_1)$:
\begin{align*}
\lefteqn{ (X_1, G_1) = \left[ \left( \{\widehat{\sigma_{1,1}}\}, \langle \ell_{1,1}^- \cdot \ell_{1,1}^+\rangle \right) \sqcup \left( \{\widehat{\sigma_{1,2}}\}, \langle \ell_{1,2}^- \cdot \ell_{1,2}^+\rangle \right) \sqcup \ldots \right.} \\
& \ldots \left. \sqcup \left( \{\widehat{\sigma_{1,n_1-1}}\}, \langle \ell_{1,n_1-1}^- \cdot \ell_{1,n_1-1}^+\rangle \right) \sqcup 
\left( \{\widehat{\sigma_{1,n_1}}\}, \langle \ell_{1,n_1}^- \cdot \ell_{1,n_1}^+\rangle \right) \right] \times \langle h_1 \rangle.
\end{align*}
As a result of this extension, each ordering $\widehat{\sigma_{1,j}}$ splits into two orderings on the quotient to be constructed that can be identified with $\sigma_{1,j}^-$ and $\sigma_{1,j}^+$, according to $h_1(\sigma_{1,j}^+) = 1$, $h_1(\sigma_{1,j}^-) = -1$. Each pair of orderings $\widehat{\sigma_{1,i}}$, $\widehat{\sigma_{1,j}}$, $1 \leq i < j \leq n_1$, gives rise to a four-element fan $\{\sigma_{1,i}^-, \sigma_{1,i}^+, \sigma_{1,j}^-, \sigma_{1,j}^+\}$, as long as $n_1 \geq 2$.

Now we repeat the whole procedure for each of the remaining polynomials $p_2, \ldots, p_m$, and therefore we construct a sequence of spaces of orderings $(X_k, G_k)$, $k \in \{1, \ldots, m\}$:
\begin{align*}
\lefteqn{ (X_k, G_k) = \left[ \left( \{\widehat{\sigma_{k,1}}\}, \langle \ell_{k,1}^- \cdot \ell_{k,1}^+\rangle \right) \sqcup \left( \{\widehat{\sigma_{k,2}}\}, \langle \ell_{k,2}^- \cdot \ell_{k,2}^+\rangle \right) \sqcup \ldots \right.} \\
& \ldots \left. \sqcup \left( \{\widehat{\sigma_{k,n_k-1}}\}, \langle \ell_{k,n_k-1}^- \cdot \ell_{k,n_k-1}^+\rangle \right) \sqcup 
\left( \{\widehat{\sigma_{k,n_k}}\}, \langle \ell_{k,n_k}^- \cdot \ell_{k,n_k}^+\rangle \right) \right] \times \langle h_k \rangle.
\end{align*}

Relabeling, if necessary, we may assume that $\alpha_{1,1}$ is the smallest real number among $\alpha_{k,1}$, $k \in \{1, \ldots, m\}$, and that $\alpha_{m, 1}$ is the largest one. Finally, in the last step of the proof we set
$$(X_0, G_0) = \left( \{\widehat{\infty^-}\}, \langle \ell_{1,1}^- \rangle \right) \sqcup (X_1, G_1) \sqcup \ldots \sqcup (X_m, G_m) \sqcup \left( \{\widehat{\infty^+}\}, \langle -\ell_{m,n_m}^+ \rangle \right),$$
where $\widehat{\infty^-}$ denotes the unique ordering of the quotient that makes $\ell_{1,1}^-$ negative, and $\widehat{\infty^+}$ denotes the unique ordering that makes $-\ell_{m,n_m}^+$ negative. From the above construction it is clear that $(X_0, G_0)$ is a quotient structure of $(X_{\Q(x)}, G_{Q(x)})$, that is that the orderings of $\Q(x)$ restrict to the orderings in $X_0$, and that $\widetilde{G} \subset G_0$. By the structure theorem for finite spaces of orderings it follows that this quotient structure is, in fact, a quotient space. \end{proof}

{\bf Remarks:} (1) Inspection of the proof readily shows that the expression of $(X_{\Q(x)}, G_{\Q(x)})$ as a profinite space of orders can have a countable index set.  For the elements of $G_{\Q(x)}$ are countable and therefore the collection of quotients can be chosen to have an increasing collection of groups whose union is  $G_{\Q(x)}$.  \\
(2) Techniques similar to the ones developed in the proof of the theorem can be used to investigate certain quotients of the space $(X_{\Q(x)}, G_{\Q(x)})$.\\
(3) The celebrated ``Lam's Open Problem B'' has a positive solution for all spaces of orderings that are profinite (\cite[Remark 5.1, Theorem 5.2]{Marshall1979}). Theorem 1 thus provides yet another proof of Lam's problem for the space of orderings $(X_{\Q(x)}, G_{\Q(x)})$. The result itself is somewhat trivial, as it is well known that the problem has a positive solution for all spaces of stability index at most 3 (\cite[Proposition 3.1 together with the beginning of Section 4]{Marshall2002}), and that the stability index of the space $(X_{\Q(x)}, G_{\Q(x)})$ is equal to 2 (\cite[Proposition VI.3.5]{ABR}), yet we believe it is worth mentioning that profiniteness yields another proof of that fact.\\
(4) An easy variant of the proof shows that if $\Q$ is replaced by $\R$ Theorem 1 still holds, although the index set will no longer be countable.

\section{Inverse limits and the pp conjecture}

Recall that, for a space of orderings $(X, G)$, a positive primitive (pp for short) formula $P(\underline{a})$ with $n$ quantifiers and $k$ parameters in $G$ is of the form
$$P(\underline{a}) = \exists \underline{t} \bigwedge_{j=1}^m p_j(\underline{t}, \underline{a}) \in D_X(1, q_j(\underline{t}, \underline{a})),$$
where $\underline{t} = (t_1, \ldots, t_n)$, $\underline{a} = (a_1, \ldots, a_k)$, for $a_1, \ldots, a_k \in G$, and $p_j(\underline{t}, \underline{a})$, $q_j(\underline{t}, \underline{a})$ are $\pm$ products of some of the $t_i$'s and $a_l$'s, $i \in \{1, \ldots, n\}$, $l \in \{1, \ldots, k\}$, for $j \in \{1, \ldots, m\}$. While speaking of the formula $P(\underline{a})$ in a subspace $Y$, we mean the formula obtained from $P(\underline{a})$ by replacing each atom $p_j(\underline{t}, \underline{a}) \in D_X(1, q_j(\underline{t}, \underline{a}))$ by $p_j(\underline{t}, \underline{a})|_Y \in D_Y(1, q_j(\underline{t}, \underline{a})|_Y)$. The following question, which can be viewed as a type of very general and highly abstract local-global principle, is known as the pp conjecture and was posed in \cite{Marshall2002}: is it true that if a pp formula holds in every finite subspace of a space of orderings, then it also holds in the whole space?

Our main goal in this section is the following theorem, first proven by Astier and Tressl in \cite[Proposition 6]{AstTre}:

\begin{theorem}[Astier, Tressl] If $(X,G) = \underleftarrow{\lim} (X_i, G_i)$, for some inverse system of spaces of orderings $(I, (X_i, G_i), F_{ij})$, and if, for all $i \in I$, the pp conjecture holds in $(X_i, G_i)$, then it also holds in $(X, G)$.
\end{theorem}

The original proof given by Astier and Tressl uses techniques from model theory, while ours utilizes only basic notions from the theory of spaces of orderings and some elementary topology. Both proofs make use of the following lemma proved by Marshall, that first appeared in print in \cite[Lemma 4]{AstTre}:

\begin{lemma}[Marshall] \label{boundcard}
Let $B(n, 0) = 1$ for $n \in \N$, and let
$$B(n, k) = 2^k 2^{2^{nk}B(n, k-1)}, \mbox{ if } k \geq 1, n \in \N.$$
Then, for every space of orderings $(X, G)$, for every $\underline{a} \in G^k$, and for every pp formula $P(\underline{y})$ with $n$ quantifiers and $k$ parameters, if $P(\underline{a})$ fails to hold in $(Z, G|_Z)$, for a finite subspace $Z$ of $(X, G)$ (or, more generally, a subspace $Z$ such that $(Z, G|_Z)$ has a finite chain length), then there is a subspace $Y$ of $(X, G)$ such that $P(\underline{a})$ fails to hold in $(Y, G|_Y)$ and $|Y| \leq B(n,k)$.
\end{lemma}

We now proceed to the proof of the theorem.

\begin{proof} Let $(X,G) = \underleftarrow{\lim} (X_i, G_i)$, and let
$$P(\underline{a}) = \exists \underline{t} \bigwedge_{j=1}^m p_j(\underline{t}, \underline{a}) \in D_X(1, q_j(\underline{t}, \underline{a})),$$
where $\underline{t} = (t_1, \ldots, t_n)$, $\underline{a} = (a_1, \ldots, a_k)$, for $a_1, \ldots, a_k \in G$, and $p_j(\underline{t}, \underline{a})$, $q_j(\underline{t}, \underline{a})$ are $\pm$ products of some of the $t_i$'s and $a_l$'s, $i \in \{1, \ldots, n\}$, $l \in \{1, \ldots, k\}$, be a pp formula that holds true on every finite subspace of $(X, G)$. Moreover, let $I_0 = \{i \in I: a_1, \ldots, a_k \in \gamma_i(G_i)\}$. It suffices to show that, for some $i \in I_0$, the formula $P(\underline{a})$, holds true in $(X_i, G_i)$ (note that, for $i \in I_0$, $a_1, \ldots, a_k \in G_i$).

Suppose, {\em a contrario,} that $P(\underline{a})$ fails in $(X_i, G_i)$, for all $i \in I_0$. Since the pp conjecture holds true in every $(X_i, G_i)$, by Lemma \ref{boundcard} for every $i \in I_0$ there exists a finite subspace of $(X_i, G_i)$ of $B$ elements, $Z_i = \{x_1^i, \ldots, x_B^i\}$, such that $P(\underline{a})$ already fails in $(Z_i, G_i|_{Z_i})$.

Let $\mathbf{x}_p^i \in \pi_i^{-1}(x_p^i)$, $p \in \{1, \ldots, B\}$, $i \in I_0$. $\{\mathbf{x}_1^i: i \in I_0\}$ is a net in the compact space $X$, and hence has a cluster point $\mathbf{x}_1$. Let $\{\mathbf{x}_1^i: i \in J_1\}$ be a net finer than $\{\mathbf{x}_1^i: i \in I_0\}$ that converges to $\mathbf{x}_1$, $J_1 \subset I_0$. Now, $\{\mathbf{x}_2^i: i \in J_1\}$ is a net that has a cluster point $\mathbf{x}_2$, and let $\{\mathbf{x}_2^i: i \in J_2\}$ be a net finer than $\{\mathbf{x}_2^i: i \in J_1\}$ that converges to $\mathbf{x}_2$, $J_2 \subset J_1$. Recursively we will eventually construct a net $\{\mathbf{x}_B^i: i \in J_B\}$ convergent to $\mathbf{x}_B$, and finer than the net $\{\mathbf{x}_B^i: i \in J_{B-1}\}$, whose cluster point is $\mathbf{x}_B$, $J_B \subset J_{B-1}$. Then, as each net refines the previous one, for $p \in \{1, \ldots, B\}$, $\mathbf{x}_p$ is the limit of the net
$$\{\mathbf{x}_p^i: i \in J_B\}.$$

Let $Y$ be the subspace of $(X, G)$ generated by $\mathbf{x}_1, \ldots, \mathbf{x}_B$. By \cite[Note 1 p. 39]{Marshall1996}, $Y$ is finite. We claim that the formula $P(\underline{a})$ fails on $(Y, G|_Y)$. Suppose that $P(\underline{a})$ holds true in $(Y, G|_Y)$ with $\underline{t} = (t_1, \ldots, t_n)$ verifying it, for $t_1, \ldots, t_n \in G$. Let
$$U = \bigcap_{j=1}^m \left( H_X(p_j(\underline{t}, \underline{a})) \cup H_X(-q_j(\underline{t}, \underline{a}))\right).$$
Clearly $Y \subset U$, and, in particular, $\mathbf{x}_1, \ldots, \mathbf{x}_B \in U$. Since each $\mathbf{x}_p$ is a limit of the net $\{\mathbf{x}_p^i: i \in J_B\}$, $p \in \{1, \ldots, B\}$, there is $i_0 \in J_B$ such that $\mathbf{x}_p^j \in U$, for all $j \succeq i_0$, $j \in J_B$, and for all $p \in \{1, \ldots, B\}$. Moreover, there is $i_1 \in J_B$ such that $t_1, \ldots, t_n \in G_j$, for all $j \succeq i_1$, $j \in J_B$. Take $i \succeq \max \{i_0, i_1\}$. Then
$$x_p^i = \pi_i(\mathbf{x}_p^i) \in \bigcap_{j=1}^m \left(H_{X_i}(p_j(\underline{t}, \underline{a})) \cup H_{X_i}(-q_j(\underline{t}, \underline{a}))\right),$$
so that $P(\underline{a})$ holds true in $(Z_i, G_i|_{Z_i})$ -- a contradiction.\end{proof}

{\bf Remarks:} (1) As an immediate consequence of the above theorem we get that the pp conjecture holds true in $(X_{\Q(x)}, G_{\Q(x)})$. This has been already shown in \cite{DicMarMir}, 
where the proof is relying on the structure of real valuations of $\Q(x)$.\\
(2) Since the pp conjecture fails for spaces of orderings of function fields of rational conic sections without rational points, for the space of orderings of $\Q(x_1, \ldots, x_n)$, $n \geq 2$, or for the space of orderings of $\R(x_1, \ldots, x_n)$, $n \geq 2$ (see \cite{GlaMara}, \cite{GlaMarb}), neither of these spaces can be profinite.\\
(3) It would be interesting to find an example of a space of orderings which is not profinite, yet which satisfies the pp conjecture.\\
(4) Since Lam's Open Problem B is implied by the pp conjecture (see \cite{Marshall2002}), we have yet another proof of the fact that the problem has an affirmative solution for the field $\Q(x)$. This adds to our discussion in Remark (2) towards the end of the previous section.

\end{document}